\documentclass[11pt]{article}

\usepackage{tikz}
\usetikzlibrary{automata,arrows.meta,positioning}

\usepackage[margin=1.15in]{geometry}
\usepackage{amsmath,amssymb,amsthm,mathtools}
\usepackage{enumitem}
\usepackage{hyperref}
\usepackage{xcolor}

\hypersetup{
  colorlinks=true,
  linkcolor=blue!60!black,
  citecolor=blue!60!black,
  urlcolor=blue!60!black
}

\newtheorem{theorem}{Theorem}[section]
\newtheorem{proposition}[theorem]{Proposition}

\newtheorem{corollary}[theorem]{Corollary}
\newtheorem{conjecture}[theorem]{Conjecture}

\theoremstyle{definition}
\newtheorem{definition}[theorem]{Definition}
\newtheorem{example}[theorem]{Example}
\theoremstyle{remark}
\newtheorem{remark}[theorem]{Remark}

\newcommand{\Irr}{\mbox{Irr}}
\newcommand{\Ind}{\mbox{Ind}}
\newcommand{\Res}{\mbox{Res}}
\newcommand{\mult}{\mbox{mult}}
\newcommand{\Hom}{\mbox{Hom}}
\newcommand{\supp}{\mbox{supp}}
\newcommand{\Spec}{\mbox{Spec}}
\newcommand{\Sch}{\mbox{Sch}}
\newcommand{\Cay}{\mbox{Cay}}

\title{On the Laplacian spectral gap of generalized pancake graphs}
\author{Sa\'ul A. Blanco\footnote{Department of Computer Science. Indiana University, Bloomington. Email: \texttt{sblancor@iu.edu}}}
\date{September 3, 2026}

\begin{document}
\maketitle

\begin{abstract}
    The generalized pancake graph $P(m,n)$ is the Cayley graph of the group of colored permutations $\mathbb{Z}_m\wr S_n=(\mathbb{Z}_m)^n\rtimes S_n$ generated by generalized prefix reversals. In this paper, we establish that, for all $m,n\geq2$, the spectral gap $\gamma(P(m,n))$ of the normalized Laplacian satisfies $\alpha_m/n\leq\gamma(P(m,n))\leq1/n$,  where $\alpha_m$ is a positive constant that depends only on $m$. As a consequence, for every fixed $m\geq2$, $\gamma(P(m,n))$ is $\Theta_m(1/n)$ as $n\to\infty$. The proof combines Cesi's semi-recursive spectral-gap inequality with a Fourier decomposition of the appropriate operators associated with a coset Schreier graph of color-position pairs. For fixed $n\geq2$, we also establish that $\gamma(P(m,n))$ is $\Theta_n(m^{-2})$ as $m\to\infty$. This disproves a conjecture of Blanco and Buehrle asserting that, for fixed $n$, the corresponding undirected generalized pancake graphs form an expander family. Additionally, we present a counterexample to a recent conjecture of Greaves and Zhu concerning equality between the spectral gaps of the full Cayley graph and the associated coset Schreier graph.
\end{abstract}


\section{Introduction}

Graphs generated by reversals and prefix reversals are of interest in computer science and mathematics. Several authors have studied their adjacency or Laplacian spectral gap, or that of related graphs~\cite{BlancoBuehrle2024, BB25, Cesi2009, ChungTobin2017, Greaves2026}. To be more specific, Cesi~\cite{Cesi2009} proved that the unnormalized spectral gap of the pancake graph is exactly 1. Later, Chung and Tobin extended Cesi's results to reversal graphs~\cite{ChungTobin2017}. Chung and Tobin also discussed the spectral gap of the burnt pancake graph. Moreover, Cesi specialized the semi-recursive method to the hyperoctahedral group tower $W(B_{n-1})\leq W(B_n)$. Indeed,~\cite[Proposition 4.2]{Cesi20} is the $m=2$ specialization of the comparison used here; however, the results from~\cite{Cesi20} concern a different class of generators. Blanco and Buehrle~\cite{BlancoBuehrle2024} established that the spectrum of this graph includes a range of integers from $\{0,\ldots,n\}\setminus \{\lfloor\frac{n}{2}\rfloor\}$. Similar results were later found for all generalized pancake graphs, both directed and undirected~\cite{BB25}. 

For $m,n\geq1$, we write $S(m,n)$ to denote the group of colored permutations $\mathbb{Z}_m\wr S_n=(\mathbb{Z}_m)^n\rtimes S_n$. We furthermore use $R(m,n)$ to denote the set of generalized prefix reversals. The generalized pancake graph $P(m,n)$ is the Cayley graph of $S(m,n)$ with respect to $R(m,n)$. In this paper, we prove a tight asymptotic bound for the normalized Laplacian spectral gap of generalized pancake graphs. In terms of the unnormalized adjacency spectral gap, the results here show that for fixed $m$, it is bounded above and below by positive constants depending only on $m$, uniformly for all $n\geq2$. However, a closed formula for the spectral gaps, adjacency or Laplacian, remains open. 

The proof follows the recursive philosophy introduced by Caputo, Liggett and Richthammer in their proof of Aldous's spectral gap conjecture~\cite{CLR10} and, independently in a related form, by Dieker~\cite{Dieker10}. We use Cesi’s representation-theoretic semi-recursive formulation of that strategy. More specifically, we apply the semi-recursive inequality from Proposition 3.2 in~\cite{Cesi2016} to the tower $S(m,n-1)\leq S(m,n)$. This inequality relates the spectral gap at level $n$ to the spectral gap at level $n-1$ and to the spectral gap of the coset permutation representation on $S(m,n)/S(m,n-1)$. This coset space is naturally identified with $\mathbb{Z}_m\times[n]$, and the associated Schreier graph records the color and position of one distinguished symbol. The same $mn$-vertex color-position graph arises as an equitable quotient that was previously studied by Blanco and Buehrle~\cite{BlancoBuehrle2024, BB25} and by Greaves and Zhu~\cite{Greaves2026}. Our contribution is to estimate the normalized Laplacian spectral gap of this color-position Schreier graph by decomposing its normalized adjacency operator into Fourier blocks, each of size $n\times n$. Combining these estimates with Cesi’s inequality in an inductive argument yields bounds for the spectral gap of the full Cayley graph on $m^nn!$ vertices. \\

\textbf{Contribution.} The main result of the paper is the following: 

\begin{theorem}[Main Theorem]\label{thm:main} If $m\geq2$ and $n\geq2$, then the normalized Laplacian spectral gap $\gamma(P(m,n))$ of $P(m,n)$ satisfies the following inequality
\[
\frac{\alpha_m}{n}\leq\gamma(P(m,n))\leq \frac{1}{n},
\] where
\[\alpha_m=\begin{cases}
    1-\cos(\pi/m)&\text{ if $m$ is odd}\\
    \min\{1-\cos(2\pi/m),2-\sqrt{2}\}&\text{ if $m$ is even}
\end{cases}.
\] In particular, for fixed $m\geq2$,  $\gamma(P(m,n))$ is $\Theta_m(1/n)$ as $n\to\infty$. 
\end{theorem} 

As a corollary, we also obtain that for fixed $n\geq2$, $\gamma(P(m,n))$ is $\Theta_n(m^{-2})$ as $m\to\infty$. More specifically, we show that
\[
\frac{2}{nm^2}\leq \gamma(P(m,n))\leq\frac{\pi^2(n+1)}{nm^2}.
\]

This bound is used to disprove a conjecture of Blanco and Buehrle~\cite[Conjecture 6.3]{BB25} asserting that the family $\{P(m,n)\}_{m>2}$, with fixed $n>2$, is an expander family.

The Main Theorem extends known results by Cesi~\cite{Cesi2009} and Chung and Tobin~\cite{ChungTobin2017} for the pancake graph case, $m=1$. Indeed, they establish that $\gamma(P(1,n))=\frac{1}{n-1}$ for $n\geq3$. Therefore, $\gamma(P(1,n))$ is $\Theta(1/n)$ as $n\to\infty$ as well. Further, $\gamma(P(1,2))=2$ since $P(1,2)$ is an edge. So the spectral gap is known for pancake graphs for all $n\geq2$. \\

\textbf{Organization. } The organization of the paper is as follows. We present the needed preliminaries and notation in Section~\ref{sec:prelim}. This includes the needed representation theory and spectral graph theory definitions. In Section~\ref{sec:Cesi}, we present Cesi's semi-recursive results, including a form in terms of a coset Schreier graph. This allows us to reduce the computations from a graph with $m^nn!$ vertices to a graph with $mn$ vertices. In Section~\ref{sec:proof} we include the proof of the Main Theorem, and as a corollary, we also prove an asymptotic bound for $\gamma(P(m,n))$ with fixed $n$. We end with some remarks, including some directions for future research. We also include an observation by Qiyuan (Alex) Gu providing a counterexample to a conjecture of Greaves and Zhu~\cite{Greaves2026}, and formulate a new conjecture. 

\section{Preliminaries}\label{sec:prelim}

Let $G$ be a finite group and let $S\subseteq G$ be a generating set. We denote the \textit{Cayley graph} of $G$ with respect to $S$ by $\Cay(G,S)$. This is the graph with vertex set $G$ and edge set $\{(x,s\cdot x):x\in G,s\in S\}$. 

Consider integers $m,n\geq1$. The set $\{1,\ldots,n\}$ is denoted by $[n]$. Let $S_n$ denote the symmetric group of degree $n$, and let $\mathbb{Z}_m=\mathbb{Z}/m\mathbb{Z}$ be the group of integers modulo $m$. We can let $S_n$ act on $(\mathbb{Z}_m)^n$ by permuting coordinates. That is, if $\sigma\in S_n,a\in(\mathbb{Z}_m)^n$ and $a_i$ denotes the $i$th component of $a$, then $(\sigma\cdot a)_j=a_{\sigma^{-1}(j)}$. Let $S(m,n)$ denote the group $\mathbb{Z}_m\wr S_n=(\mathbb{Z}_m)^n\rtimes S_n$. This group is called the group of \textit{colored permutations}. 

For $1\leq i\leq n$, let $r_i\in S_n$ be the permutation that reverses the first $i$ positions in one-line notation, and leaves every other position fixed. For example, $r_3(541236)=145236$. We set $r_1=e$, the identity permutation in $S_n$.

For $m\geq 2$ and $n\geq1$, the set of \textit{generalized prefix reversals}  $R(m,n)$ is the set of all elements of the form $\{r^{\varepsilon}_i=(c^{\varepsilon}_i,r_i)\in S(m,n):1\leq i\leq n\text{ and }\varepsilon\in\{-,+\}\}$, where $c^{\varepsilon}_i$ is the $n$-dimensional vector with the first $i$ components are all $\varepsilon1$, and the remaining components are 0. That is,
\[
c^{\varepsilon}_i=(\underbrace{\varepsilon1,\varepsilon1,\ldots,\varepsilon1}_{i\text{ components}},0,\ldots,0)\in(\mathbb{Z}_m)^n,
\] where $\varepsilon1$ denotes $-1$ or $+1$, respectively. 

We note that if $m=2$ then $r^+_i=r^-_i$ for all $1\leq i\leq n$. Additionally, we set $R(1,n)=\{r_i:2\leq i\leq n\}$, and so $|R(1,n)|=n-1, |R(2,n)|=n$, and $|R(m,n)|=2n$ for $m\geq3$.

The \textit{generalized pancake graph} or \textit{prefix-reversal graph}, denoted by $P(m,n)$, is the Cayley graph $\Cay(G,S)$ with $G=S(m,n)$ and $S=R(m,n)$. The graphs $P(1,n)$ and $P(2,n)$ are referred to as the \textit{pancake graph} and \textit{burnt pancake graph}, respectively. Since the set of prefix reversals generates $S_n$ in the case $m=1$ and one can cycle through all different colors in the first position, $P(m,n)$ is connected. Moreover, these graphs are weakly pancyclic if $m=1,2$ (see Kanevsky and Feng~\cite{Kanevsky1995} and Blanco, Buehrle and Patidar~\cite{BBP19}). Blanco and Buehrle~\cite{BB23} later proved that if $m\geq3$ is odd, then $P(m,n)$ is also weakly pancyclic, whereas if $m\geq3$ is even, then $P(m,n)$ possesses cycles of all even length $\ell\geq\min\{m,6\}$. 

Let $\mathcal{G}=(V,E,w_E)$ be a weighted graph with weight function $w_E:E\to\mathbb{R}_{\geq0}$. The adjacency matrix $A_\mathcal{G}$ of $\mathcal{G}$ is the matrix whose entries are of the form
\[
A_{\mathcal{G}}(u,v)=\sum_{\substack{e\in E\\e\text{ joins }u,v}}w_E(e), 
\] and the degree $d_v$ of $v\in V$ is the sum $d_v=\sum_{u\in V}A_\mathcal{G}(v,u)$. We are only concerned with graphs that do not have isolated vertices; that is, $d_v\neq0$ for all $v\in V$. Let $D_\mathcal{G}$ be the diagonal matrix whose $v$th entry is the degrees $d_v$ of vertex $v$. The \textit{unnormalized Laplacian} of $\mathcal{G}$ is defined as $\mathcal{L}_\mathcal{G}=D_\mathcal{G}-A_\mathcal{G}$. Moreover, the \textit{normalized adjacency matrix} of $\mathcal{G}$ is defined as $\mathcal{NA}_\mathcal{G}=D^{-1/2}_\mathcal{G}A_{\mathcal{G}}D^{-1/2}_\mathcal{G}$, and the \textit{normalized Laplacian} of $\mathcal{G}$ is defined as $\mathcal{NL}_\mathcal{G}=D_\mathcal{G}^{-1/2}\mathcal{L}_\mathcal{G}D_\mathcal{G}^{-1/2}=I-\mathcal{NA}_\mathcal{G}$. If $\mathcal{G}$ is $d$-regular, then $\mathcal{NA}_\mathcal{G}=\frac{1}{d}A_\mathcal{G}$, $\mathcal{L}_\mathcal{G}=dI-A_\mathcal{G}$, and $\mathcal{NL}_\mathcal{G}=I-\frac{1}{d}A_\mathcal{G}$. The graphs treated in this paper are all regular. In particular,  Cayley graphs $\Cay(G,S)$ are $d$-regular with $d=|S|$.

We write $0=\nu_1(\mathcal{NL}_\mathcal{G})\leq \nu_2(\mathcal{NL}_\mathcal{G})\leq \cdots$ to denote the eigenvalues of $\mathcal{NL}_\mathcal{G}$, and $\lambda_1(\mathcal{NA}_\mathcal{G})\geq\lambda_2(\mathcal{NA}_\mathcal{G})\geq\cdots$ to denote the eigenvalues of $\mathcal{NA}_\mathcal{G}$. 

The \textit{normalized Laplacian spectral gap} of $\mathcal{G}$ is defined as $\nu_2(\mathcal{NL}_\mathcal{G})$, and we denote it by $\gamma(\mathcal{G})$. The \textit{unnormalized adjacency spectral gap} of $\mathcal{G}$ is defined as $\lambda_1(A_\mathcal{G})-\lambda_2(A_\mathcal{G})$. Moreover, for $d$-regular graphs, $\lambda_1(A_\mathcal{G})=d$, and then 
\[
\gamma(\mathcal{G})=\nu_2(\mathcal{NL}_\mathcal{G})=\frac{1}{d}\nu_2(\mathcal{L}_{\mathcal{G}})=1-\frac{1}{d}\lambda_2(A_\mathcal{G}).
\]

\textbf{Important convention:} When we say \textit{spectral gap} of $\mathcal{G}$ and write $\gamma(\mathcal{G})$, we mean the normalized spectral gap of the Laplacian of $\mathcal{G}$, $\nu_2(\mathcal{NL}_\mathcal{G})$. That is, $\gamma(\mathcal{G})=\nu_2(\mathcal{NL}_\mathcal{G})$.

\subsection{Rayleigh quotient and variational principle}\label{sec:variational} If $A$ is a real symmetric or complex Hermitian matrix, its \textit{Rayleigh quotient} at a non-zero vector $\mathbf{v}$ is defined as
\[
R_A(\mathbf{v})=\frac{\langle \mathbf{v},A\mathbf{v}\rangle}{\langle \mathbf{v},\mathbf{v}\rangle}.
\]

The \textit{variational principle} (also referred to as \textit{min-max theorem}, or \textit{Courant-Fischer theorem}, or \textit{Rayleigh-Ritz theorem}) gives that
\[
\lambda_{\min}(A)\leq R_A(\mathbf{v})\leq \lambda_{\max}(A)\text{ for any }\mathbf{v}\neq\mathbf{0}.\]

\subsection{Basic notation} 

Throughout the paper, we use $G$ to denote a finite group, $[n]$ to denote the set of positive integers from $1$ to $n$, $X$ a finite set, and $\ell^2(X;\mathbb{F})=\{f:X\to\mathbb{F}\}$ the space of $\mathbb{F}$-valued functions on $X$ (most of the time, we only need $\mathbb{F}=\mathbb{R}$). The space $\ell^2(X;\mathbb{F})$ is endowed with the inner product 
\[
\langle f,g\rangle=\sum_{x\in X}\overline{f(x)}g(x).
\]

For self-adjoint operators $P_1$ and $P_2$, the notation $P_1\succeq P_2$ means that $P_1-P_2$ is positive semi-definite. In other words, $\langle x,P_1x\rangle \geq \langle x,P_2x\rangle$ for all $x$. (This is the so-called \textit{Loewner order}.)

\begin{definition}
Let $G$ act on a finite set $X$, and let $S\subseteq G$ be a finite generating set.  The \textit{Schreier graph} $\Sch(G,X,S)$ is the generator-labeled multigraph (so loops and parallel edges are allowed) with vertex set $X$ and, for every $x\in X$ and $s\in S$, one edge from $x$ to $s\cdot x$ labeled by $s$. If $S=S^{-1}$, its adjacency matrix is symmetric, and the edges corresponding to an element and its inverse may be identified to regard the graph as undirected. By convention, each generator contributes $1$ to the degree, even if it fixes a vertex.
\end{definition}

The action $G\curvearrowright X$ gives rise to a permutation representation $\rho_X:G\longrightarrow U(\ell^2(X;\mathbb{C}))$ defined by $\bigl(\rho_X(g)f\bigr)(x) =f(g^{-1}\cdot x)$. The \textit{adjacency operator} $A_{X,S}$ of the Schreier graph $\Sch(G,X,S)$ is
\[
   (A_{X,S}f)(x)= \sum_{s\in S}f(s^{-1}\cdot x)=\sum_{s\in S}(\rho_X(s)f)(x).
\]
Thus
\[
   (A_{X,S}f)(x)
   =
   \sum_{s\in S}f(s^{-1}\cdot x).
\]
The \textit{unnormalized Schreier Laplacian} is $L_{X,S}=|S|I-A_{X,S}$.

\begin{example}
    Let $G=S_3$ act naturally on $X=[3]$. Let $a,b$ denote the transpositions $(1 \;2)$ and $(2\;3)$, respectively, and $S=\{a,b\}$. Then $\Cay(G,S)$ is a $6$-cycle and the Schreier graph $\Sch(G,X,S)$ has $3$ vertices and the following edges: $a$ joins vertices $1$ and $2$, and $b$ joins vertices $2$ and $3$. Moreover, $a$ fixes vertex $3$, and $b$ fixes vertex $1$. The resulting $\Sch(G,X,S)$ is depicted in Figure~\ref{fig:Sch_ex}.

\begin{figure}[h]
\begin{center}
    \begin{tikzpicture}[scale=0.9,
    vertex/.style={circle, draw, minimum size=0.9cm, inner sep=0pt},
    edge label/.style={font=\large}
]

\node[vertex] (v1) at (0,0) {1};
\node[vertex] (v2) at (4,0) {2};
\node[vertex] (v3) at (8,0) {3};

\draw (v1) -- node[edge label, above] {$a$} (v2);
\draw (v2) -- node[edge label, above] {$b$} (v3);

\draw[->] (v1) to[out=140,in=40,looseness=6] node[edge label, above] {$b$} (v1);
\draw[->] (v3) to[out=140,in=40,looseness=6] node[edge label, above] {$a$} (v3);

\end{tikzpicture}
\caption{The Schreier graph $\Sch(S_3,[3], S)$ with $a=(1\;2),b=(2\;3)$ and $S=\{a,b\}$.}
\label{fig:Sch_ex}
\end{center}
\end{figure}
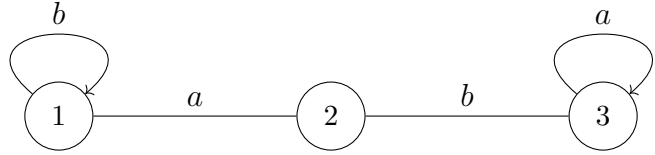
\end{example}

 \subsection{Representation theory notation}\label{sec:rep} This paper is written for an audience familiar with the foundations of spectral graph theory and random walks, but not necessarily with representation theory. In the interest of readability, we now describe the representation theory notation that is used throughout the paper. For a general reference, please see~\cite{James_Liebeck_2001}. While the notation is standard in representation theory, we follow Cesi~\cite{Cesi2009, Cesi2016, Cesi20} in the exposition. Let $\pi:G\to U(V_\pi)$ be a unitary, finite-dimensional complex representation, where $V_\pi$ is a complex vector space with some inner product and $U(V_\pi)$ its unitary group. If $H\leq G$, then we use the following notation. 

\begin{itemize}
    \item $\Irr(G)$ denotes the set of equivalence classes of irreducible, finite-dimensional complex representations of $G$.
    \item $\pi\vert_H=\Res^G_H\pi$ denotes the restriction of $\pi$ to the subgroup $H$.
    \item $V_\pi^H$ denotes the $H$-fixed space $\{v\in V_\pi:\pi(h)v=v\text{ for all }h\in H\}$.
    \item $\mathbf{1}_H$ denotes the one-dimensional trivial representation of $H$.
    \item $\Ind_H^G\mathbf{1}_H$ denotes the induced representation $\rho_{G/H}$ of $G$ from $\mathbf{1}_H$. More specifically, this is the permutation representation of $G/H$.
    \item $\mult_\pi(\Pi)$ denotes the multiplicity of the irreducible representation $\pi$ in $\Pi$.
\end{itemize}

Notice that \textbf{Frobenius reciprocity} yields that  
\[
\Hom_G\bigl(\Ind_H^G\mathbf{1}_H,\pi \bigr)\cong \Hom_H\bigl(\mathbf{1}_H,\pi\vert_H\bigr).
\] Moreover, since $\Hom_H\bigl(\mathbf{1}_H,\pi\vert_H\bigr)\cong V_\pi^H$, by taking dimensions, it follows that $\mult_\pi(\Ind_H^G\mathbf{1}_H)=\dim V_\pi^H$. Since $\Ind_H^G\mathbf{1}_H$ can be realized as the permutation representation of $G$ on the coset space $G/H$, it follows that $\pi$ occurs in $\Ind_H^G\mathbf{1}_H$ if and only if $\pi\vert_H$ contains the trivial representation $\mathbf{1}_H$.

\subsection{Representation Laplacian}
We use $\mathbb{F}G$ to denote the \textit{group algebra of $G$}. That is, the $\mathbb{F}$-vector space of formal sums of the form 
$\displaystyle{
\sum_{g\in G}w_gg,
}$ with multiplication extended linearly by the group operation. Moreover, the \textit{support} $\supp(w)$ of $w$ is defined as $\supp(w)=\{g\in G:w_g\neq0\}$.  Furthermore, if $w_g\in\mathbb{R}_{\geq0}$ for all $g$, then $w$ is called \textit{positive}. In addition, if $w=\sum_{g\in G}w_gg\in \mathbb{F}G$, then $w^*$ denotes the formal sum $\displaystyle{\sum_{g\in G}\overline{w_g}g^{-1}}$. We say that $w\in \mathbb{F}G$ is \textit{symmetric} if $w=w^*$. In $\mathbb{R}G$, $w$ being symmetric is equivalent to $w_g=w_{g^{-1}}$ for all $g\in G$.

If $w=\sum_{g\in G}w_gg$, then the \textit{representation Laplacian} $\Delta_G(w,\pi)$ is defined as
\[
\Delta_G(w,\pi)=\sum_{g\in G}w_g(I_{V_\pi}-\pi(g)).
\] Here $I_{V_\pi}$ denotes the identity operator on the representation space $V_\pi$. 

Let $\mathbb FG^{(s)}=\{w\in\mathbb{F}G:w=w^*\}$, and let
\[
   \mathbb R_+G^{(s)}
   :=
   \left\{
      \sum_{g\in G}w_g g:
      w_g\in\mathbb R,\;
      w_g\ge0,\;
      w_g=w_{g^{-1}}
   \right\}
\]
denote the set of positive, symmetric elements of the real group algebra. If $w\in\mathbb RG^{(s)}$, then $\Delta_G(w,\pi)$ is self-adjoint (cf. Cesi~\cite[Proposition 2.1]{Cesi2016}). If $w\in\mathbb R_+G^{(s)}$, then the following identity is standard. In the representation-Laplacian setting, it appears, for example, as
Equation~(2.2) in the proof of Cesi~\cite[Proposition~2.1]{Cesi2016}
\begin{equation}\label{eq:energy}
\langle\Delta_G(w,\pi)v,v\rangle=\frac{1}{2}\sum_{g\in G}w_g\|\pi(g)v-v\|^2,
\end{equation} from which it follows that $\Delta_G(w,\pi)$ is positive semi-definite, i.e., $\Delta_G(w,\pi)\succeq 0$. Equation~(\ref{eq:energy}) can be seen as the representation-theoretic analogue of the usual Dirichlet-form identity for reversible Markov chains
\cite[Lemma~13.6]{Levin2017}.
 
If $w$ is symmetric and $\pi$ is unitary, then
$\Delta_G(w,\pi)$ is self-adjoint.  We therefore list its
eigenvalues, repeated according to multiplicity, in non-decreasing
order:
\[
   \nu_1\bigl(\Delta_G(w,\pi)\bigr)
   \le
   \nu_2\bigl(\Delta_G(w,\pi)\bigr)
   \le\cdots\le
   \nu_{d_\pi}\bigl(\Delta_G(w,\pi)\bigr),
\]
where $d_\pi=\dim V_\pi$. If $w$ is positive symmetric, then $ 0\leq \nu_1\bigl(\Delta_G(w,\pi)\bigr)$. If, in addition, $\supp(w)$ generates $G$, then $\ker\Delta_G(w,\pi)=V_\pi^G$. Hence, zero is an eigenvalue precisely when $V_\pi^G\neq\{0\}$, and its multiplicity is $\dim V_\pi^G$.

Notice that the representation Laplacian respects direct sums.  Indeed, if
$\pi=\pi_1\oplus\pi_2$, then $\pi(g)=\pi_1(g)\oplus\pi_2(g)$ for all $g\in G$. Therefore,
$\Delta_G(w,\pi)=\Delta_G(w,\pi_1)\oplus\Delta_G(w,\pi_2)$.
More generally, if $\displaystyle{\pi\cong\bigoplus_\tau m_\tau\,\tau}$, then
$\displaystyle{\Delta_G(w,\pi) \cong \bigoplus_\tau \left(I_{m_\tau}\otimes\Delta_G(w,\tau)\right)}$.

\subsubsection{Spectral gap of the representation Laplacian} Following Cesi, an eigenvalue $\lambda$ of $\Delta_G(w,\pi)$ is called \textit{trivial} if its eigenspace is contained in $V_\pi^G$; otherwise it is called \textit{nontrivial}.  Every trivial eigenvalue is necessarily zero. The \textit{spectral gap} of the pair $(w,\pi)$, with $w\in\mathbb{R}_+G^{(s)}$, is defined as 
\[
\psi_G(w,\pi)=\min\bigl\{\lambda\in\Spec(\Delta_G(w,\pi)):\lambda\text{ is nontrivial}\bigr\},
\] with the convention that $\min\emptyset=+\infty$. Then the \textit{spectral gap} of $w$ is the minimum $\psi_G(w,\pi)$ over all irreducible representations $\pi\in\Irr(G)$. Namely,
\[
\psi_G(w)=\min\{\psi_G(w,\pi):\pi\in\Irr(G)\}.
\]

\subsubsection{Spectral gap of the representation Laplacian and spectral gap of the Laplacian}
    
Let $\rho_G$ denote the left regular representation of $G$ on $\ell^2(G;\mathbb{C})$. 

The weighted Cayley adjacency operator is given by
\[
   (A_wf)(x)=\sum_{g\in G}w_gf(g^{-1}\cdot x)=\sum_{g\in G}w_g(\rho_G(g)f)(x),
\]
and hence its unnormalized Laplacian is
\[
   L_w
   =
   d_wI-A_w
   =
   \sum_{g\in G}w_g\bigl(I-\rho_G(g)\bigr)
   =
   \Delta_G(w,\rho_G),
\] where $d_w=\sum_{g\in G}w_g$. Thus the representation Laplacian with the regular representation is exactly the weighted Cayley-graph Laplacian. In particular, if $Q$ is a symmetric generating set of $G$, and $w=\sum_{q\in Q}q\in\mathbb{R}_+G^{(s)}$, then $A_w$ and $L_w$ are the unnormalized adjacency operator and unnormalized Laplacian of $\mathcal{G}=\Cay(G,Q)$, and so $L_w=\Delta_G(w,\rho_G)=d_wI-A_\mathcal{G}=\mathcal{L}_\mathcal{G}$ in this case. Similarly, if $H\leq G$ and $\mathbf{1}_H$ is the trivial representation of $H$, then $\Delta_G(w,\Ind^G_H\mathbf{1}_H)$ is the unnormalized Laplacian of $\Sch(G,G/H,Q)$.

The \textit{unnormalized spectral gap} of $L_w$ is defined as $d_w-\lambda_2(A_w)$. The \textit{normalized spectral gap} is obtained by dividing by $d_w$.

Notice that the identity $\Delta_G(w,\rho_G)=L_w$ holds for every positive symmetric group-algebra element $w$. Furthermore, since the left regular representation contains every irreducible
representation of $G$ and the representation Laplacian respects direct sums, one has
\[
   \psi_G(w)=\psi_G(w,\rho_G).
\]

Since $w$ is symmetric and nonnegative, $A_w$ is self-adjoint and $L_w$ is positive semi-definite. We can therefore order their eigenvalues as follows

\[ 
d_w=\lambda_1(A_w)\geq \lambda_2(A_w)\geq\cdots
 \quad\text{ and }\quad
0=\nu_1(L_w)\leq \nu_2(L_w)\leq\cdots
\]

If one additionally assumes $\langle\operatorname{supp}(w)\rangle=G$, then the weighted Cayley graph is connected. Thus, 
$\ker L_w$ consists precisely of the constant functions and so $0=\nu_1(L_w)<\nu_2(L_w)$. That is, $\nu_2(L_w)$ is the first positive Laplacian eigenvalue. So if $\supp(w)$ generates $G$, then 
\[
\psi_G(w)=\psi_G(w,\rho_G)=\nu_2(L_w)=d_w-\lambda_2(A_w).
\]

Although $\psi_G(w)$ could equivalently be defined simply using the left regular representation, the formulation as a minimum over irreducible representations is useful for the semi-recursive argument, where the irreducible representations are separated according to whether their restriction to a subgroup $H$ contains the trivial representation. We make this explicit in the next section. 

\section{Cesi's semi-recursive inequality and its Schreier form}\label{sec:Cesi}

Following Cesi~\cite[Equation (3.1)]{Cesi2016}, let
$
\Gamma(G)=\bigl\{w\in \mathbb{R}G^{(s)}:\Delta_G(w,\pi)\succeq 0 \text{ for all }\pi\in\Irr(G)\bigr\}.
$ As a consequence of Equation (\ref{eq:energy}), $\mathbb{R}_+G^{(s)}\subseteq\Gamma(G)$.

The following is one of the main tools used to prove our results. We include a short proof in the present notation.

\begin{proposition}[Cesi~\cite{Cesi2016}, Proposition 3.2]\label{prop:cesi} Let $H\leq G$ be finite groups. Let $w\in\mathbb{R}_+G^{(s)}$ and $z\in \mathbb{R}_+H^{(s)}$ with $w-z\in\Gamma(G)$. Then
\[
\psi_G(w)\geq  \min\left\{
      \psi_H(z),\,
      \min_{\substack{\pi\in\operatorname{Irr}(G)\\
                      \mathbf{1}_H\subseteq \pi|_H}}
      \psi_G(w,\pi)
   \right\}.
\]
\end{proposition}
\begin{proof}
    Let $\pi\in\Irr(G)$. Recall that if $v\in V_\pi$ satisfies $\pi(h)v=v$ for every $h\in H$, then $v$ is called a \textit{$H$-fixed vector}. There are two possibilities: (i) $\pi$ has a non-zero $H$-fixed vector, or (ii) it does not. If $\pi$ has a non-zero fixed $H$-vector, then $\pi$ is accounted for in 
    \[\min_{\substack{\pi\in\operatorname{Irr}(G)\\
                      \mathbf{1}_H\subseteq \pi|_H}}
      \psi_G(w,\pi)
      \] since $\mathbf{1}_H\subseteq \pi|_H$ is equivalent to $V_\pi^H\neq\{0\}$. For (ii), if $\pi$ has no non-zero $H$-fixed vector (so $V_\pi^H=\{0\}$), by Maschke's theorem, the restricted representation is completely reducible:
\[
   \pi|_H
   \cong
   \bigoplus_{\tau\in\operatorname{Irr}(H)}
   m_\tau\,\tau.
\]
The multiplicity of the trivial representation $\mathbf 1_H$ in $\pi|_H$ is $\dim V_\pi^H$.  Hence $m_{\mathbf 1_H}=0$, and
\[
   \pi|_H
   \cong
   \bigoplus_{\substack{\tau\in\operatorname{Irr}(H)\\
                        \tau\not\cong\mathbf 1_H}}
   m_\tau\,\tau.
\]
The representation Laplacian respects direct sums, so
$\displaystyle{
   \Delta_H(z,\pi|_H)
   \cong
   \bigoplus_{\substack{\tau\in\operatorname{Irr}(H)\\
                        \tau\not\cong\mathbf 1_H}}
   \left(
      I_{m_\tau}\otimes\Delta_H(z,\tau)
   \right)}$.
 Hence $\psi_H(z,\tau)= \lambda_{\min}\bigl(\Delta_H(z,\tau)\bigr)$. Now, since $\displaystyle{\psi_H(z) = \min_{\sigma\in\operatorname{Irr}(H)} \psi_H(z,\sigma) \leq \psi_H(z,\tau)}$, it follows that $\lambda_{\min}\bigl(\Delta_H(z,\tau)\bigr) \geq \psi_H(z)$. Equivalently, $\Delta_H(z,\tau) \succeq \psi_H(z)I_{V_\tau}$. Moreover, 
\[
   \Delta_G(w,\pi)
   =
   \Delta_H(z,\pi|_H)+\Delta_G(w-z,\pi)
   \succeq
   \psi_H(z)I_{V_\pi}, 
\] since $w-z\in\Gamma(G)$. Thus, every irreducible representation whose restriction to $H$ does not contain the trivial representation has spectral gap at least $\psi_H(z)$. The remaining irreducible representations are precisely those satisfying $\mathbf{1}_H\subseteq\pi\vert_H$. Taking the minimum over all $\pi\in\Irr(G)$ proves the proposition.
\end{proof}

As a corollary, we obtain the following form of Cesi's semi-recursive inequality, written in the language of Schreier graphs.

\begin{corollary}[Schreier form of Proposition~\ref{prop:cesi}]\label{cor:cesi}
    With the same assumptions as in Proposition~\ref{prop:cesi}, let $\beta_{G/H}(w)=\psi_G(w,\rho_{G/H})$. Assume further that $\langle\supp(w)\rangle=G$ and $\langle\supp(z) \rangle=H$. Then,
    \begin{equation}\label{eq:cesi-schreier}
    \min\{\psi_H(z),\beta_{G/H}(w)\}\leq \psi_G(w)\leq \beta_{G/H}(w).
    \end{equation}
\end{corollary}
\begin{proof}
    Let $\rho_X$ denote the permutation representation of $G$ on $\ell^2(X;\mathbb{C})$ given by
   $\bigl(\rho_X(g)f\bigr)(x)= f(g^{-1}\cdot x)$, with $f\in\ell^2(X;\mathbb{C})$. The \textit{weighted Schreier graph} $\Sch(G,X,w)$ associated with the action $G\curvearrowright X$ and weights $(w)_{g\in G}$ is the graph with vertex set $X$ and adjacency operator given by
\[
(A_{X,w}f)(x)=\sum_{g\in G}w_g f(g^{-1}\cdot x)=\sum_{g\in G}w_g(\rho_X(g)f)(x). 
\]

Let $S=\supp(w)$ and $d_w=\sum_{g\in G}w_g$. Then, the unnormalized Laplacian of $\Sch(G,X,w)$ is $L_{X,w}= d_wI-A_{X,w}$. By definition of the representation Laplacian,
\[
   L_{X,w} = d_wI-A_{X,w}=\sum_{s\in S}w_s(I-\rho_X(s))=\sum_{g\in G}w_g(I-\rho_X(g))=\Delta_G(w,\rho_X).
\]

Now let $X=G/H$, and let $\rho_{G/H}$ denote the corresponding coset permutation representation.  The representation $\rho_{G/H}$ is naturally isomorphic to $\Ind_H^G\mathbf 1_H$. As a consequence of Frobenious reciprocity (see Section~\ref{sec:rep}), $\mult_\pi(\rho_{G/H})=\dim V_\pi^H$ for every $\pi\in\Irr(G)$. Thus, $\rho_{G/H}\cong\bigoplus_{\pi\in\Irr(G)} (\dim V^H_{\pi})\pi$. Therefore $\pi$ occurs in $\rho_{G/H}$ if and only if $V_\pi^H\neq\{0\}$ if and only if $\mathbf 1_H\subseteq\pi\vert_H$. Since the representation Laplacian respects direct sums, 
\[
   \psi_G(w)=\min_{\pi\in\Irr(G)}\psi_G(w,\pi)\leq\min_{\substack{\pi\in\operatorname{Irr}(G)\\\mathbf 1_H\subseteq\pi\vert_H}}
   \psi_G(w,\pi)=\beta_{G/H}(w),
\]
which gives the upper half of~(\ref{eq:cesi-schreier}). In particular, $\beta_{G/H}(w)$ is the first positive eigenvalue of the weighted Schreier Laplacian on $G/H$.

For the lower part of the inequality~(\ref{eq:cesi-schreier}), Cesi's semi-recursive inequality Proposition~\ref{prop:cesi} gives $\psi_G(w)\geq \min\{\psi_H(z),\beta_{G/H}(w)\}$.
\end{proof}

\begin{remark}
Cesi's \cite[Proposition~3.2]{Cesi2016} is slightly more general: it does not require the supports of $w$ and $z$ to generate their respective groups and uses his notion of the smallest nontrivial eigenvalue.  We impose the generating assumptions because all Cayley and Schreier graphs in the present application are connected. In our setting the relevant gaps are simply the first positive eigenvalues of the Laplacian.
\end{remark}

\begin{example} To illustrate the corollary, let $G=S_3$ and  $H=\langle(1\;2)\rangle\cong S_2$. Furthermore, let $z=(1\;2)\in\mathbb R_+H^{(s)}$ and $w=(1\;2)+(2\;3)\in\mathbb R_+G^{(s)}$. So, $w-z=(2\;3)\in \Gamma(G)$. Notice that this setting satisfies the hypothesis of Corollary~\ref{cor:cesi}.

The Cayley graph $\Cay(H,\{(1\;2)\})$ is an edge and thus has $\psi_H(z)=2$.  The Schreier adjacency operator is 
\[
   A_{G/H,w}=
   \begin{pmatrix}
      1&1&0\\
      1&0&1\\
      0&1&1
   \end{pmatrix}.
\]
Therefore, $\displaystyle{L_{G/H,w}=2I-A_{G/H,w}=
   \begin{pmatrix}
      1&-1&0\\
      -1&2&-1\\
      0&-1&1
   \end{pmatrix}}$. The eigenvalues of this unnormalized Laplacian are $0,1,3$, so $\beta_{G/H}(w)=1$. Since $\min\{\psi_H(z),\beta_{G/H}(w)\}=1$, the spectral gap of the unnormalized Laplacian of $\Cay(S_3,\{(1\;2),(2\;3)\})$ is 1. This can easily be verified to be true since $\Cay(S_3,\{(1\;2),(2\;3)\})$ is a hexagon. 
\end{example}

\section{Proof of the Main Theorem}\label{sec:proof}

In this section, we prove our Main Theorem. 

There are some special cases that one needs to handle. If $m=1$, then it is known that the adjacency spectral gap of the pancake graph is 1. This is Cesi's main result in~\cite{Cesi2009} (see also Chung and Tobin~\cite{ChungTobin2017} who prove a more general result). Therefore, we assume that $m\geq2$. If $m=2$, then $r^+_i=r^{-}_i$ for every $1\leq i\leq n$, and so the corresponding Cayley graph is $n$-regular. If $m>2$, then the Cayley graph is $2n$-regular. Since this is the only distinction in the argument, we shall simply use $q_m$ to distinguish between the degree of the graph when $m=2$ and $m>2$
\[
q_m=\begin{cases}
    1&\text{ if }m=2\\
    2&\text{ if }m>2
\end{cases}.
\] So the corresponding Cayley graph $\Cay\bigl(S(m,n),R(m,n)\bigr)$ is $q_mn$-regular. 

For readability, we divide the proof of the Main Theorem into four steps. We first present a brief overview of each of the steps. 

\begin{description}
   \item[Step 1: Schreier graph formulation.] The coset space $S(m,n)/S(m,n-1)$ is naturally identified with $X_{m,n}=\mathbb Z_m\times[n]$. To see this, select a distinguished symbol $n$ and let $x_0=(0,n)\in X_{m,n}$ (meaning color zero and position $n$). The embedded subgroup $S(m,n-1)$ is exactly the stabilizer of $x_0$. Therefore, the map $gS(m,n-1)\mapsto g\cdot x_0$ is an $S(m,n)$-equivariant bijection. The image is the orbit of $x_0$, which consists of all $mn$ color-position pairs. In particular, the Schreier graph records only the color and position of one distinguished element from $[n]$ and has $mn$ vertices. We can then use the inequality from Corollary~\ref{cor:cesi}.
   \item[Step 2: Decomposition of the Schreier graph adjacency operator.] We decompose the Schreier adjacency operator utilizing a discrete Fourier transform. Then the Schreier-gap calculation reduces to estimating the eigenvalues of $m$ matrices of size $n\times n$, rather than those of a single $mn\times mn$ matrix. 
   \item[Step 3: Bounding the spectral gap of the Schreier graph.] This shall be done in three cases depending on $r\in\mathbb{Z}_m$. The case $r=0$, the case where $r\neq0$ does not have order $2$, and the case where $r\neq0$ has order $2$. 
   \item[Step 4: Induction.] The final proof uses induction with the bounds found in the previous step and the inequality from Corollary~\ref{cor:cesi}. 
\end{description}

\subsection{Step 1} Notice that we can embed $S(m,n-1)$ into $S(m,n)$ by utilizing the following identification:
\[
S(m,n-1)=\{(a,\sigma)\in S(m,n):a_n=0,\sigma(n)=n\}.
\]

Moreover, let 
\[
w_{m,n-1}=\sum_{r\in R(m,n-1)}r\in\mathbb{R}_+S(m,n-1)^{(s)}\quad \text{ and }\quad w_{m,n}=\sum_{r\in R(m,n)}r\in\mathbb{R}_+S(m,n)^{(s)}.
\]
Notice that $w_{m,n}-w_{m,n-1}\in\Gamma(S(m,n))$. Indeed,
\[
w_{m,n}-w_{m,n-1}=\begin{cases}
    r^+_n &\text{ if }m=2,\\
    r^+_n+r^-_n&\text{ if }m\geq3.
\end{cases}
\]

If $m=2$, the new generator $r^+_n$ is an involution, so it is positive symmetric. Moreover, if $m\geq3$, then $r^-_n=(r_n^+)^{-1}$, so $r^+_n+r^-_n$ is positive symmetric. Therefore, $w_{m,n}-w_{m,n-1}\in\mathbb{R}_+S(m,n)^{(s)}\subseteq\Gamma(S(m,n))$.

Let $\psi_{m,n-1}=\psi_{H}(z)$, $\psi_{m,n}=\psi_G(w)$, and $\beta_{m,n}=\beta_{G/H}(w)$. Applying Corollary~\ref{cor:cesi} with $G=S(m,n)$, $H=S(m,n-1)$, $w=w_{m,n}$, and $z=w_{m,n-1}$, we obtain
\[
\min\{\psi_{m,n-1},\beta_{m,n}\}\leq\psi_{m,n}\leq\beta_{m,n}.
\]
We now need to estimate $\beta_{m,n}$.

Let $\Sch(m,n)=\Sch(G,X,S)$, where $G=S(m,n)$, $X=S(m,n)/S(m,n-1)$, $S=R(m,n)$; furthermore, let $X_{m,n}=\mathbb{Z}_m\times [n]$.

\subsection{Step 2} Let us denote by $\mathcal{NA}^{\mathrm{Sch}}_{m,n}$ the normalized Schreier graph adjacency operator of $\Sch(m,n)$. Let $(c,p)\in X_{m,n}$, where the first component $c$ is referred to as the \textit{color} and the second component $p$ is referred to as the \textit{position}. Let
\[
E_m=\begin{cases}\{+\}&\text{ if }m=2\\
\{-,+\}&\text{ if }m\geq3
\end{cases}.
\]

If $\varepsilon\in E_m$, $r^{\varepsilon}_i\in R(m,n)$,  and $x=(a_x,p_x)\in V(\Sch(m,n))$ then $a_x\in\mathbb{Z}_m$ and $p_x\in [n]$. The action is given by $r^{\varepsilon}_i\cdot (a_x,p_x)=(a_x+\varepsilon1,i+1-p_x)$ if $p_x\leq i$ and $r^{\varepsilon}_i\cdot (a_x,p_x)=(a_x,p_x)$ if $p_x>i$. The addition is performed modulo $m$. 

By definition, the normalized adjacency operator $\mathcal{NA}^{\mathrm{Sch}}_{m,n}$ on $\ell^2(X_{m,n};\mathbb{C})$ satisfies $\mathcal{NA}^{\mathrm{Sch}}_{m,n}(x,y)=\frac{1}{q_mn}|\{s\in R(m,n):s\cdot x=y\}|$, where $s\cdot x$ is as described above. Therefore, since $R(m,n)^{-1}=R(m,n)$,
\[
(\mathcal{NA}^{\mathrm{Sch}}_{m,n}F)(a,p)=\frac{1}{q_mn}\sum_{s\in R(m,n)}F(s^{-1}\cdot (a,p))=\frac{1}{q_mn}\sum_{s\in R(m,n)}F(s\cdot (a,p)),
\] and the normalized Laplacian of $\Sch(m,n)$ is then
$\displaystyle{
    \mathcal{NL}^{\mathrm{Sch}}_{m,n}=I-\mathcal{NA}^{\mathrm{Sch}}_{m,n}
}$, and so $\beta_{m,n}=(q_mn)\gamma(\Sch(m,n))=(q_mn)\nu_2\bigl(\mathcal{NL}^{\mathrm{Sch}}_{m,n}\bigr)$. Equivalently,
\begin{equation}\label{eq:beta_mn_in_terms_of_A}
    \beta_{m,n}=q_mn\left(1-\lambda_2\left(\mathcal{NA}^{\mathrm{Sch}}_{m,n}\right)\right)
\end{equation}

First, let us describe the operator $\mathcal{NA}^{\mathrm{Sch}}_{m,n}$.

\begin{proposition}\label{prop:A_in_terms_of_F}
    The operator $\mathcal{NA}^{\mathrm{Sch}}_{m,n}$ satisfies
    \[
    (\mathcal{NA}^{\mathrm{Sch}}_{m,n}F)(a,p)=\frac{1}{q_mn}\left(q_m(p-1)F(a,p)+\sum_{\varepsilon\in E_m}\sum_{q=1}^{n+1-p}F(a+\varepsilon1,q)\right),
    \] where the addition is done modulo $m$.
\end{proposition}
\begin{proof}
    This follows from the definition of the action $s\cdot(a,p)$ with $s\in R(m,n)$ and $(a,p)\in V(\Sch(m,n))=X_{m,n}$. Indeed, the sum 
    \[
    \sum_{s\in R(m,n)}F(s\cdot(a,p))=\sum_{\varepsilon\in E_m}\sum_{i=1}^nF(r^\varepsilon_i\cdot(a,p))
    \] can be decomposed into elements with $i< p$ and $i\geq p$. Therefore,
    \begin{align*}
        (\mathcal{NA}^{\mathrm{Sch}}_{m,n}F)(a,p) &=\frac{1}{q_mn}\sum_{\varepsilon\in E_m}\sum_{i=1}^nF(r^\varepsilon_i\cdot(a,p))\\
        &=\frac{1}{q_mn}\left(q_m(p-1)F(a,p)+\sum_{\varepsilon\in E_m}\sum_{i=p}^{n}F(a+\varepsilon1,i+1-p)\right)\\
        &=\frac{1}{q_mn}\left(q_m(p-1)F(a,p)+\sum_{\varepsilon\in E_m}\sum_{q=1}^{n+1-p}F(a+\varepsilon1,q)\right),
    \end{align*} as desired. 
\end{proof}

Now we shall utilize the discrete Fourier transform in our computations. In particular, we refer to Luong~\cite{Luong2009}. Another classical reference is Terras~\cite{Terras1999}. 

Specializing Luong's formulas from Example 4.1.1, Equations (4.3) and (4.4), we obtain the following discrete Fourier transform for $F\in\ell^2(X_{m,n};\mathbb{C})$
\[
\widehat{F}_r(p)=\frac{1}{\sqrt{m}}\sum_{a=0}^{m-1}\omega^{-ra}F(a,p)\quad\text{ with }\quad r\in\mathbb{Z}_m
\text{ and }\omega=e^{2\pi i/m}.\] 

Its inverse transform is given by
$\displaystyle{
F(a,p)=\frac{1}{\sqrt{m}}\sum_{r=0}^{m-1}\omega^{ra}\widehat{F}_r(p).
}$

Define the \textit{color Fourier transform}
\[
   \mathcal F_{\mathrm{col}}:
   \ell^2(\mathbb Z_m\times[n];\mathbb C)
   \longrightarrow
   \bigoplus_{r=0}^{m-1}\ell^2([n];\mathbb C)
\quad\text{ by }\quad
   \mathcal F_{\mathrm{col}}F
   =
   \bigl(\widehat F_0,\ldots,\widehat F_{m-1}\bigr).
\]

Then, we have the following proposition. 

\begin{proposition}\label{prop:A_in_terms_of_T}  For $r\in\{0,\ldots,m-1\}$, let us write $c_r=\cos\left(2\pi r/m\right)$. Furthermore, let $T_{n,r}$ act on $\ell^2([n];\mathbb C)$ by
\[
   (T_{n,r}f)(p)
   =
   \frac{1}{n}(p-1)f(p)
   +
   \frac{1}{n}c_r
   \sum_{q=1}^{n+1-p}f(q).
\]
Then $\mathcal{F}_{\mathrm{col}}$ is unitary and 
$\displaystyle{
   \mathcal F_{\mathrm{col}}\,
   \mathcal NA_{m,n}^{\mathrm{Sch}}\,
   \mathcal F_{\mathrm{col}}^{-1}
   =
   \bigoplus_{r=0}^{m-1}T_{n,r}.
}$ As a consequence, \[\Spec\bigl(\mathcal{NA}^{\mathrm{Sch}}_{m,n}\bigr)=\bigsqcup_{r=0}^{m-1}\Spec(T_{n,r}),\] where the spectra are regarded as multisets. 
\end{proposition}

\begin{proof}
For each $r\in\{0,\ldots,m-1\}$, define
$\displaystyle{
   J_r:\ell^2([n];\mathbb C)
   \longrightarrow
   \ell^2(\mathbb Z_m\times[n];\mathbb C)
}$
by
\[
   (J_rf)(a,p)
   =
   \frac1{\sqrt m}\omega^{ra}f(p).
\]
Borrowing terminology from Fourier analysis, the image $V_r=\operatorname{im}(J_r)$ is called the \textit{$r$th color mode}. Notice that the subspaces $V_r$ are mutually orthogonal.  Indeed, if
$f,g\in\ell^2([n];\mathbb{C})$, then
\begin{equation*}
   \langle J_rf,J_sg\rangle
   =\frac1m
   \sum_{a=0}^{m-1}\sum_{p=1}^n
   \overline{\omega^{ra}f(p)}
   \,\omega^{sa}g(p)\\
   =\left(
      \frac1m\sum_{a=0}^{m-1}\omega^{(s-r)a}
   \right)
   \langle f,g\rangle\\
   =\delta_{r,s}\langle f,g\rangle, 
\end{equation*}
where $\delta_{r,s}$ is the Kronecker delta and the last equality follows from the orthogonality of the characters of $\mathbb Z_m$:
\[
   \frac1m\sum_{a=0}^{m-1}\omega^{(s-r)a}=
   \begin{cases}
      1,&r=s\\
      0,&r\ne s
   \end{cases}.
\]
Therefore, each $J_r$ is an isometry and the spaces $V_0,\ldots,V_{m-1}$ are mutually orthogonal. Moreover, each $V_r$ has dimension $n$, while $\dim\ell^2(\mathbb Z_m\times[n])=mn$. So, it follows that
\[
   \ell^2(\mathbb Z_m\times[n];\mathbb C)= \bigoplus_{r=0}^{m-1}V_r.
   \]
   
Equivalently, every function
$F\in\ell^2(\mathbb Z_m\times[n])$ has a unique expansion
\[
   F(a,p)=
   \frac1{\sqrt m}
   \sum_{r=0}^{m-1}
   \omega^{ra}\widehat{F}_r(p),
\quad \text{ where }\quad
   \widehat{F}_r(p)=
   \frac1{\sqrt m}
   \sum_{a=0}^{m-1}
   \omega^{-ra}F(a,p)
\]
is the $r$th Fourier coefficient of $F$.

We next determine how
$\mathcal{NA}_{m,n}^{\mathrm{Sch}}$ acts on $V_r$.  Proposition~\ref{prop:A_in_terms_of_F} yields that
\begin{equation}\label{eq:A_in_terms_of_F}
   \bigl(\mathcal NA_{m,n}^{\mathrm{Sch}}F\bigr)(a,p)
   =\frac{1}{q_mn}\left(q_m(p-1)F(a,p) +
   \sum_{\varepsilon\in E_m}
   \sum_{q=1}^{n+1-p}F(a+\varepsilon1,q)\right),
\end{equation}
Take $F=J_rf$, so that
$\displaystyle{
   F(a,p)=\frac1{\sqrt m}\omega^{ra}f(p).
}$
Moreover, for $\varepsilon\in E_m$,
\[
   F(a+\varepsilon1,q) = \frac1{\sqrt m} \omega^{r(a+\varepsilon1)}f(q) = \frac1{\sqrt m} \omega^{ra}\omega^{r\varepsilon1}f(q).
\]
Substituting this into~(\ref{eq:A_in_terms_of_F}) gives
\[
   \bigl(\mathcal NA_{m,n}^{\mathrm{Sch}}J_rf\bigr)(a,p)=
   \frac{1}{q_mn}\left(\frac1{\sqrt m}\omega^{ra}
   q_m(p-1)f(p)
   +
   \frac1{\sqrt m}\omega^{ra}
   \left(
      \sum_{\varepsilon\in E_m}\omega^{r\varepsilon1}
   \right)
   \sum_{q=1}^{n+1-p}f(q)\right).
\]

For $m\geq3$, $E_m=\{-,+\}$ and so
\[
   \sum_{\varepsilon\in E_m}\omega^{r\varepsilon1}=\omega^r+\omega^{-r}=2
   \cos\left(\frac{2\pi r}{m}\right)=q_m\cos\left(\frac{2\pi r}{m}\right).
\]
For $m=2$, $E_2=\{+\}$ and so
\[
   \omega^r=(-1)^r=\cos\left(\frac{2\pi r}{2}\right)=q_m\cos\left(\frac{2\pi r}{2}\right).
\]
Thus, in every case,
$\displaystyle{
   \sum_{\varepsilon\in E_m}\omega^{r\varepsilon1} =
   q_mc_r}.
$

It follows that
\[
\begin{aligned}
   \bigl(\mathcal NA_{m,n}^{\mathrm{Sch}}J_rf\bigr)(a,p)
   &=
   \frac{1}{q_mn}\left(\frac1{\sqrt m}\omega^{ra}
   \left[
      q_m(p-1)f(p)
      +
      q_mc_r
      \sum_{q=1}^{n+1-p}f(q)
   \right]\right)\\
   &=
   \frac1{\sqrt m}\omega^{ra}(T_{n,r}f)(p)\\
   &=
   \bigl(J_rT_{n,r}f\bigr)(a,p).
\end{aligned}
\]
Therefore,
\begin{equation}\label{eq:NA_in_terms_of_JT}
   \mathcal NA_{m,n}^{\mathrm{Sch}}J_r
   =
   J_rT_{n,r}.
\end{equation}

There are two consequences of Equation~(\ref{eq:NA_in_terms_of_JT}). First, $\mathcal NA_{m,n}^{\mathrm{Sch}}$ maps $V_r$ into $V_r$, so every color mode $r$ is invariant.  Second, under the unitary identification $J_r:\ell^2([n])\longrightarrow V_r$, the restriction of the Schreier normalized adjacency operator to $V_r$ is exactly $T_{n,r}$. In other words,
$
   J_r^{-1}
   \left(
      \left.
      \mathcal NA_{m,n}^{\mathrm{Sch}}
      \right\vert_{V_r}
   \right)
   J_r
   =
   T_{n,r}.
$

Finally, define
\[
   J:
   \bigoplus_{r=0}^{m-1}\ell^2([n];\mathbb C)
   \longrightarrow
   \ell^2(\mathbb Z_m\times[n];\mathbb C)
\quad\text{ by }\quad
   J(f_0,\ldots,f_{m-1})
   =
   \sum_{r=0}^{m-1}J_r(f_r).
\]
Equivalently,
$\displaystyle{
   \bigl(J(f_0,\ldots,f_{m-1})\bigr)(a,p)
   =
   \frac1{\sqrt m}
   \sum_{r=0}^{m-1}\omega^{ra}f_r(p)
}$, which is the Fourier inversion formula. Thus $J=\mathcal F_{\mathrm{col}}^{-1}$. Further, the identity $\langle J_rf,J_sg\rangle=\delta_{r,s}\langle f,g\rangle$ shows that $J$ preserves inner products. Indeed, 
    $\displaystyle{
   \langle J\mathbf f,J\mathbf g\rangle= \sum_{r,s=0}^{m-1}\langle J_rf_r,J_sg_s\rangle=
   \sum_{r=0}^{m-1}
   \langle f_r,g_r\rangle
}$. Since
$\displaystyle{
   \ell^2(\mathbb Z_m\times[n];\mathbb C)
   =
   \bigoplus_{r=0}^{m-1}V_r
   }$, it follows that $J$ is unitary. Hence
$\mathcal F_{\mathrm{col}}=J^{-1}$ is unitary as well. Moreover, using~(\ref{eq:NA_in_terms_of_JT}),
\[
   \mathcal NA_{m,n}^{\mathrm{Sch}}
   J(f_0,\ldots,f_{m-1})=
   \sum_{r=0}^{m-1}
   \mathcal NA_{m,n}^{\mathrm{Sch}}J_rf_r=
   \sum_{r=0}^{m-1}
   J_rT_{n,r}f_r=
   J\bigl(T_{n,0}f_0,\ldots,T_{n,m-1}f_{m-1}\bigr).
\]
Hence $\displaystyle{
   J^{-1}\mathcal{NA}_{m,n}^{\mathrm{Sch}}J\left(=\mathcal{F}_{\mathrm{col}}\,
   \mathcal{NA}_{m,n}^{\mathrm{Sch}}\,
   \mathcal{F}_{\mathrm{col}}^{-1}\right)=
   \bigoplus_{r=0}^{m-1}T_{n,r}}$. 
\end{proof}

\subsection{Step 3} 

Recall that $\beta_{m,n}=q_mn\gamma(\Sch(m,n))=q_mn\left(1-\lambda_2\bigl(\mathcal{NA}^{\mathrm{Sch}}_{m,n}\bigr)\right)$. We also remark that each matrix $T_{n,r}$ as defined above is real symmetric. 

One can think of the normalized $\mathcal{NA}^{\mathrm{Sch}}_{m,n}$ as a random-walk transition operator. Since the generating set is symmetric, $\mathcal{NA}^{\mathrm{Sch}}_{m,n}$ is self-adjoint. Since $\Sch(m,n)$ is connected, the largest eigenvalue of $\mathcal{NA}^{\mathrm{Sch}}_{m,n}$ is 1, with eigenvector $\mathbf{1}$ (the vector whose all entries are 1). Moreover, this eigenvalue is unique. Notice that 
\[
(T_{n,0}\mathbf{1})(p)=\frac{1}{n}\left((p-1)+\sum_{q=1}^{n+1-p}1\right)=\frac{1}{n}\left(p-1+(n+1-p)\right)=1.
\] So 1 is an eigenvalue of $T_{n,0}$ with eigenvector $\mathbf{1}$.

Proposition~\ref{prop:A_in_terms_of_T} gives that 
$\displaystyle{\Spec\bigl(\mathcal{NA}^{\mathrm{Sch}}_{m,n}\bigr)=\bigsqcup_{r=0}^{m-1}\Spec(T_{n,r})}$. Since the unique largest eigenvalue $1$ occurs in $T_{n,0}$, we have that 

\begin{equation}\label{eq:lambda_2(A)_in_terms_of_T}
    \lambda_2\left(\mathcal{NA}^{\mathrm{Sch}}_{m,n}\right)=\max\left\{\lambda_2(T_{n,0}),\max_{1\leq r\leq m-1}\{\lambda_{\max}(T_{n,r})\}\right\}.
\end{equation}

We shall now bound $\lambda_2\left(\mathcal{NA}^{\mathrm{Sch}}_{m,n}\right)$ using the previous equality, and shall consider three cases.

\subsubsection{$r=0$ case}\label{sec:r=0}

Let $\rho_{[n]}:S_n\to U(\ell^2([n];\mathbb{C}))$ be given by $(\rho_{[n]}(\sigma)f)(p)=f(\sigma^{-1}(p))$. Since $c_r=1$ in this case, it follows that 
\begin{align*}
\frac{1}{n}\left(\sum_{i=1}^{n}(\rho_{[n]}(r_i)f)(p)\right)&=\frac{1}{n}\left((p-1)f(p)+\sum_{i=p}^{n}f(i+1-p)\right)\\
&=\frac{1}{n}\left((p-1)f(p)+\sum_{q=1}^{n+1-p}f(q)\right)\\
    &=(T_{n,0}f)(p),
\end{align*}
where $r_1$ is the identity. The first equality follows because $r_i(p)=p$ for $i<p$, while $r_i(p)=i+1-p$ if $i\geq p$. Therefore,
\[
    \lambda_2\left(T_{n,0}\right)=\lambda_2\left(\frac{1}{n}\sum_{i=1}^n\rho_{[n]}(r_i)\right).
    \]

 Gunnells, Scott and Walden~\cite[Proposition~4.1]{GSW07}
compute the spectrum of the unnormalized operator $\sum_{i=1}^n\rho_{[n]}(r_i)$ and show that its two largest eigenvalues are $n$ and $n-1$ for $n\geq3$. Therefore, for $n\ge3$, $\lambda_2(T_{n,0})=\frac{n-1}{n}$, and hence $1-\lambda_2(T_{n,0})=\frac{1}{n}$. Cesi~\cite{Cesi2009} uses their calculation as an ingredient in proving that the ordinary pancake graph $P(1,n)$ has unnormalized adjacency spectral gap equal to $1$.

For $n=2$,
\[
   T_{2,0}
   =
   \frac12
   \begin{pmatrix}
      1&1\\
      1&1
   \end{pmatrix},
\]
so its eigenvalues are $1$ and $0$.  Thus $1-\lambda_2(T_{2,0})=1\ge\frac{1}{2}$, and so $1-\lambda_2(T_{n,0})\geq\frac{1}{n}$ for all $n\geq2$.

\subsubsection{Non-zero $r$ of order other than two case}\label{sec:m_odd} In this case, $r\neq0$ and $2r\not\equiv0\pmod m$, and so $c_r\neq-1$. We estimate $\lambda_{\max}(T_{n,r})$ by utilizing the following elementary bound.  Let $T$ be
a matrix and let $\lambda$ be one of its eigenvalues, with eigenvector $\mathbf{v}\ne0$.  Choose the component $\mathbf{v}(p)$ of $\mathbf{v}$ that has the largest absolute value. That is, $|\mathbf{v}(p)|=\max_{q}|\mathbf{v}(q)|$. Then
\[
\begin{aligned}
   |\lambda|\,|\mathbf{v}(p)|=
   \left|\sum_q T(p,q)\mathbf{v}(q)\right|\leq
   \sum_q|T(p,q)|\,|\mathbf{v}(q)|\leq
   \left(\sum_q|T(p,q)|\right)|\mathbf{v}(p)|.
\end{aligned}
\]
Therefore, $|\lambda|\leq \max_p\sum_q|T(p,q)|$.
Applying this to $T_{n,r}$ gives
\[
\begin{aligned}
   \sum_{q=1}^n|T_{n,r}(p,q)|
   \leq
   \frac{1}{n}\left((p-1)+|c_r|(n+1-p)\right)
   \leq
   \frac{1}{n}(n-1+|c_r|)
\end{aligned}.
\]
Thus $\lambda_{\max}(T_{n,r}) \leq 1-\frac{1-|c_r|}{n}$. In the case $r\neq 0$ and $r\neq m/2$, it follows that 
\[
1-|c_r|\geq \begin{cases}
    1-\cos\left(\frac{\pi}{m}\right)&\text{if $m$ is odd}\\
    1-\cos\left(\frac{2\pi}{m}\right)&\text{if $m$ is even and $r\neq m/2$}
\end{cases}.
\]

\subsubsection{Non-zero $r$ of order two case}\label{sec:m_even} In this case, $2r\equiv0\pmod m$, and so $c_{m/2}=-1$. Unfortunately, the argument from Section~\ref{sec:m_odd} does not work since in this case $|c_r|=1$. Notice that since $c_{m/2}=-1$,
\begin{equation}\label{eq:T_in_terms_of_O}
    (T_{n,\frac{m}{2}}f)(p)=\frac{1}{n}(p-1)f(p)-\frac{1}{n}\sum_{q=1}^{n+1-p}f(q)=\frac{1}{n}\sum_{i=1}^{n}(O_if)(p),
\end{equation} where $(O_if)(p)=(-1)^{\delta_{i\geq p}}f(r_i\cdot p)$, where $\delta_{i\geq p}$ is 1 if $i\geq p$ and 0 otherwise. It is easy to see that $O_i$ is an involution, unitary, and self-adjoint. Moreover, notice that

\begin{align*}
    \|f-O_if\|^2_2&=\langle f-O_if,f-O_if\rangle=\|f\|_2^2-2\langle f, O_if\rangle+\|O_if\|_2^2\\&=2\|f\|_2^2-2\langle f,O_if\rangle=2\langle f,(I-O_i)f\rangle,
\end{align*} from which it follows that $\langle f,(I-O_i)f\rangle=\frac{1}{2}\|f-O_if\|^2_2$. Therefore, substituting into (\ref{eq:T_in_terms_of_O}), it follows that
\begin{equation}\label{eq:f_dot_product_in_terms_of_Os}
    \left\langle f,\left(I-T_{n,\frac{m}{2}}\right)f\right\rangle=\frac{1}{2n}\sum_{i=1}^n\|f-O_if\|^2_2.
\end{equation}

Since $(O_if)(p)=f(p)$ if $i<p$ and $(O_if)(p)=-f(i+1-p)$ if $i\geq p$, we have
\[
\frac{1}{2}\sum_{i=1}^n\|f-O_if\|^2_2=\frac{1}{2}\sum_{i=1}^n\sum_{p=1}^{i}|f(p)+f(i+1-p)|^2.
\]
Notice that for each value of $1\leq q\leq \frac{n}{2}$, the double sum includes the following terms: $|f(q)+f(q)|^2$ (when $i=2q-1$ and $p=q$), $|f(q)+f(n+1-q)|^2$ (when $i=n$ and $p=q$), and $|f(n+1-q)+f(q)|^2$ (when $i=n$ and $p=n+1-q$). Therefore, 

\begin{align*}
\frac{1}{2}\sum_{i=1}^n\sum_{p=1}^{i}|f(p)+f(i+1-p)|^2&\geq \frac{1}{2}\sum_{q=1}^{\lfloor n/2\rfloor}(4|f(q)|^2+2|f(q)+f(n+1-q)|^2)+\\&\quad\quad\quad+\frac{1}{2}\delta_{n\text{ odd}}\left(4\left|f\left(\frac{n+1}{2}\right)\right|^2\right),
\end{align*} where $\delta_{n \text{ odd}}=1$ if $n$ is odd and 0 otherwise. We now apply an elementary inequality that can be derived by observing that
\[
4|x|^2+2|x+y|^2-(4-2\sqrt{2})(|x|^2+|y|^2)=2(1+\sqrt{2})|x+(\sqrt{2}-1)y|^2\geq0.
\] In particular, $4|x|^2+2|x+y|^2\geq (4-2\sqrt{2})(|x|^2+|y|^2)$.

Therefore, if $n$ is even, then 
\begin{align*}
    \frac{1}{2}\sum_{q=1}^{\lfloor n/2\rfloor}(4|f(q)|^2+2|f(q)+f(n+1-q)|^2)&\geq \frac{1}{2}\sum_{q=1}^{\lfloor n/2\rfloor}(4-2\sqrt{2})(|f(q)|^2+|f(n+1-q)|^2)\\&=\frac{4-2\sqrt{2}}{2}\|f\|^2_2.
\end{align*}

If $n$ is odd, then
\begin{align*}
    \frac{1}{2}\sum_{q=1}^{\lfloor n/2\rfloor}(4|f(q)|^2+2|f(q)+f(n+1-q)|^2)+\\+\frac{1}{2}\left(4\left|f\left(\frac{n+1}{2}\right)\right|^2\right)
    &\geq \sum_{q=1}^{\lfloor n/2\rfloor}\frac{4-2\sqrt{2}}{2}(|f(q)|^2+|f(n+1-q)|^2)+\\&\qquad+\frac{4-2\sqrt{2}}{2}\left|f\left(\frac{n+1}{2}\right)\right|^2\\
    &=\frac{4-2\sqrt{2}}{2}\|f\|^2_2.
\end{align*}

Therefore, for all $n$, 
\begin{equation}\label{eq:f-O_inequality}
    \frac{1}{2}\sum_{i=1}^n\|f-O_if\|^2_2\geq (2-\sqrt{2})\|f\|^2.
\end{equation} 

Combining~(\ref{eq:f_dot_product_in_terms_of_Os}) with~(\ref{eq:f-O_inequality}), we obtain
$\displaystyle{
 \left\langle f,\left(I-T_{n,\frac{m}{2}}\right)f\right\rangle\geq\frac{4-2\sqrt{2}}{2n}\|f\|^2_2}$, or equivalently,
\[
T_{n,\frac{m}{2}}\preceq\left(1-\frac{2-\sqrt2}{n}\right)I, 
\] from which it follows that $\displaystyle{\lambda_{\max}\left(T_{n,\frac{m}{2}}\right)\leq 1-\frac{2-\sqrt{2}}{n}}$.

Summarizing the results from the last three subsections, we have the following.

\begin{proposition}\label{prop:beta>=}
    For $m\geq2$ and $n\geq2$, $\displaystyle{\beta_{m,n}\geq q_m\alpha_m}$.
\end{proposition}
\begin{proof}
 Notice that the case $r=0$ satisfies
$\displaystyle{
   1-\lambda_2\left(T_{n,0}\right)\geq\frac{1}{n}.
}$
For every non-zero color $r$, the estimates from Section~\ref{sec:m_odd} and~\ref{sec:m_even} give
\[
   1-\lambda_{\max}(T_{n,r})
   \geq\frac{\alpha_m}{n}\quad\text{ with }\quad \alpha_m=
    \begin{cases}
        1-\cos(\pi/m)&\text{ if $m$ is odd}\\
        \min\{1-\cos(2\pi/m),2-\sqrt{2}\}&\text{ if $m$ is even}
    \end{cases},
\]
Since $\alpha_m\le1$, the $r=0$ estimate also implies $\displaystyle{1-\lambda_2(T_{n,0})
   \ge\frac{\alpha_m}{n}}$. Therefore, it follows that
\[
\begin{aligned}
   \beta_{m,n}=q_mn\left(1-\lambda_2\left(\mathcal NA_{m,n}^{\mathrm{Sch}}
   \right)\right)
   &=
   q_mn\min\left\{
      1 -\lambda_2\left(T_{n,0}\right),
      \min_{1\leq r<m}
      \left\{1-\lambda_{\max}\left(T_{n,r}\right)\right\}
   \right\}\\
   &\geq q_m\alpha_m,
\end{aligned}
\] as desired. 
\end{proof}

\subsection{Step 4}

We now use induction to prove the Main Theorem.
\begin{proof}[Proof of Theorem~\ref{thm:main} (the Main Theorem)]
For the lower bound, we prove by induction that $\psi_{m,n}\geq q_m\alpha_m$. For the base case $n=1$, notice that $P(2,1)$ is an edge with unnormalized spectral gap equal to $\psi_{2,1}=2$.  Moreover, if $m\geq3$, then $P(m,1)$ is a cycle, which has unnormalized spectral gap $\psi_{m,1}=2(1-\cos(2\pi/m))$. Since $0<\pi/m<2\pi/m\leq 2\pi/3$, for odd $m$, $1-\cos(\pi/m)\leq 1-\cos(2\pi/m)$, as the cosine function is decreasing on $[0,2\pi/3]$. So $\psi_{m,1}\geq q_m\alpha_m$ for all $m\geq2$.  

Having established the case $n=1$ with $m\geq2$, let us assume that $\psi_{m,n-1}\geq\alpha_mq_m$. The Schreier form of Cesi's semi-recursive formula, Corollary~\ref{cor:cesi}, yields that
\begin{equation}\label{eq:mn}
\min\{\psi_{m,n-1},\beta_{m,n}\}\leq\psi_{m,n}\leq\beta_{m,n}.
\end{equation} By induction hypothesis, $\psi_{m,n-1}\geq q_m\alpha_m$. Moreover, Proposition~\ref{prop:beta>=} gives that $\beta_{m,n}\geq q_m\alpha_m$. Therefore, Corollary~\ref{cor:cesi} yields
\[
\psi_{m,n}\geq\min\{\psi_{m,n-1},\beta_{m,n}\}\geq q_m\alpha_m,
\] and after normalizing, $\gamma(P(m,n))\geq \alpha_m/n$.

To prove the upper bound, we utilize the variational principle from Section~\ref{sec:variational} and we only need to find a non-zero vector $\mathbf{v}$ and compute its Rayleigh quotient $R_{T_{n,r}}(\mathbf{v})$. Let us recall that 
\[
   (T_{n,r}f)(p)
   =
   \frac{1}{n}(p-1)f(p)
   +
   \frac{1}{n}c_r
   \sum_{q=1}^{n+1-p}f(q).
\] For $1\leq i\leq n$, we let $\displaystyle{\mathbf{e}_i(p)=\begin{cases}1&\text{ if } i=p\\0&\text{ otherwise}\end{cases}}\in\ell^2([n];\mathbb{C})$. Notice that 
\[
\langle \mathbf{e}_n,T_{n,r}\mathbf{e}_n\rangle=\sum_{p=1}^n\overline{\mathbf{e}_n(p)}(T_{n,r}\mathbf{e}_n)(p)=(T_{n,r}\mathbf{e}_n)(n)=\frac{1}{n}(n-1)+\frac{1}{n}c_r\sum_{q=1}^1\mathbf{e}_n(q)=\frac{1}{n}(n-1),
\] for $n\geq2$. Therefore, the variational principle gives that

\[
\lambda_{\max}(T_{n,r})\geq R_{T_{n,r}}(\mathbf{e}_n)=\frac{\langle \mathbf{e}_n,T_{n,r}\mathbf{e}_n\rangle}{\langle\mathbf{e}_n,\mathbf{e}_n\rangle}=\frac{1}{n}(n-1),
\]and so $\lambda_{\max}(T_{n,r})\geq \frac{n-1}{n}$. Moreover, since by (\ref{eq:lambda_2(A)_in_terms_of_T}), $\lambda_2(\mathcal{NA}^{\mathrm{Sch}}_{m,n})=\max\{\lambda_2(T_{n,0}),\max_{1\leq r<m}\{\lambda_{\max}(T_{n,r})\}\}$, it follows that if $r\neq 0$, then $\lambda_2(\mathcal{NA}^{\mathrm{Sch}}_{m,n})\geq \lambda_{\max}(T_{n,r})\geq\frac{1}{n}(n-1)$. Therefore,
\[
\beta_{m,n}=q_mn\bigl(1-\lambda_2(\mathcal{NA}^{\mathrm{Sch}}_{m,n})\bigr)\leq q_mn\left(1-\frac{n-1}{n}\right)=q_m.
\]

From Corollary~\ref{cor:cesi},  $\psi_{m,n}\leq \beta_{m,n}$, and so $\displaystyle{\gamma(P(m,n))=\frac{\psi_{m,n}}{q_mn}\leq \frac{1}{n}}$.
\end{proof}

\begin{remark}
    The case $m=1$ was studied by Cesi~\cite{Cesi2009} and later by Chung and Tobin~\cite{ChungTobin2017}. Indeed, they show that the following equality holds: $\gamma(P(1,n))=\frac{1}{n-1}$ for $n\geq3$.
\end{remark}

\subsection{Asymptotics with fixed $n$}

As a corollary of the Main Theorem, we can also provide asymptotics for $\gamma(P(m,n))$ with fixed $n$.

\begin{corollary}\label{cor:Theta_n}
    For all $n,m\geq2$, 
    $\displaystyle{
    \frac{2}{nm^2}\leq \gamma(P(m,n))\leq\frac{\pi^2(n+1)}{nm^2}}$.
\end{corollary}
\begin{proof}
    Notice that the Main Theorem gives that $\gamma(P(m,n))\geq \alpha_m/n$. Since $1-\cos(x)= 2\sin^2(x/2)$, it follows that for $0\leq x\leq \pi$,
    \[
    1-\cos(x)=2\sin^2\left(\frac{x}{2}\right)\geq \frac{2x^2}{\pi^2}.
    \]

    Let us consider the parity of $m$:
    \begin{itemize}
        \item $m$ odd: Then $\displaystyle{\alpha_m=1-\cos\left(\frac{\pi}{m}\right)\geq \frac{2}{m^2}}$.
        \item $m$ even: Then $\alpha_m=\min\{1-\cos(2\pi/m),2-\sqrt{2}\}$. Notice that $\displaystyle{1-\cos\left(\frac{2\pi}{m}\right)\geq \frac{8}{m^2}\geq \frac{2}{m^2}}$, and that $2-\sqrt{2}\geq 2/m^2$ for all $m\geq2$.
    \end{itemize}
    So, it follows that for all $m\geq2$,
    \[
    \alpha_m\geq\frac{2}{m^2},\quad\text{ and hence }\quad \gamma(P(m,n))\geq\frac{2}{nm^2}\quad\text{ for all }m,n\geq2.
    \]  

    To establish an upper bound for $\gamma({P(m,n)})$, we shall upper bound $\beta_{m,n}$ and then use Corollary~\ref{cor:cesi}. For this, take $r=1$. Then, $c_1=\cos(2\pi/m)$ and 
    \[(T_{n,1}f)(p)=\frac{1}{n}\left((p-1)f(p)+c_1\sum_{q=1}^{n+1-p}f(q)\right).
    \]

    Let $\mathbf{1}_n\in\ell^2([n];\mathbb{C})$ denote the vector whose entries are all $1$.  Then
    \[
    (T_{n,1}\mathbf{1}_n)(p)=\frac{p-1}{n}+\frac{c_1(n+1-p)}{n}.
    \] 

    Therefore, 
    \[
    R_{T_{n,1}}(\mathbf{1}_n)=\frac{\langle \mathbf{1}_n,T_{n,1}\mathbf{1}_n\rangle}{\langle\mathbf{1}_n,\mathbf{1}_n\rangle}=\frac{1}{n}\sum_{p=1}^n\left(\frac{p-1}{n}+\frac{c_1(n+1-p)}{n}\right)=\frac{(n-1)+c_1(n+1)}{2n}.
    \] It follows that 
    $\displaystyle{\lambda_{\max}(T_{n,1})\geq 1-\frac{n+1}{2n}(1-c_1)=1-\frac{n+1}{2n}(1-\cos(2\pi/m)).
    }$ Since $r=1$ is a non-zero color mode, $\lambda_2(\mathcal{NA}^{\mathrm{Sch}}_{m,n})\geq\lambda_{\max}(T_{n,1})$. Thus
    \[
    \lambda_2(\mathcal{NA}^{\mathrm{Sch}}_{m,n})\geq 1-\frac{n+1}{2n}(1-\cos(2\pi/m)), \text{ and so }1-\lambda_2(\mathcal{NA}^{\mathrm{Sch}}_{m,n})\leq\frac{n+1}{2n}(1-\cos(2\pi/m)).
    \] Using $1-\cos(x)\leq\frac{x^2}{2}$, we have that 
    \[
    1-\lambda_2(\mathcal{NA}^{\mathrm{Sch}}_{m,n})\leq\frac{n+1}{2n}\frac{4\pi^2}{2m^2}=\frac{\pi^2(n+1)}{nm^2}.
    \]
Thus, 
\[
\gamma(P(m,n))=\frac{\psi_{m,n}}{q_mn}\leq\frac{\beta_{m,n}}{q_mn}=1-\lambda_2(\mathcal{NA}^{\mathrm{Sch}}_{m,n})\leq\frac{\pi^2(n+1)}{nm^2}.
\] Combining the upper and lower bounds, we conclude that, for every fixed $n\geq2$, $\gamma(P(m,n))$ is $\Theta_n(m^{-2})$ as $m\to\infty$.
\end{proof}

\begin{remark}
Recall that the \textit{expansion ratio} of a finite
$d$-regular graph $\mathcal G=(V,E)$ is
\[
   h(\mathcal G)= \min_{\substack{\emptyset\neq U\subseteq V\\|U|\leq |V|/2}} \frac{|\partial U|}{|U|},
\]
where $\partial U$ is the set of edges having one endpoint in $U$ and the other in $V\setminus U$.  The \textit{discrete Cheeger inequality}
gives $h(\mathcal G)\leq d\sqrt{2\gamma(\mathcal G)}$.
For $m>2$, the graph $P(m,n)$ is $2n$-regular.  Hence the preceding
Corollary~\ref{cor:Theta_n} yields
\[
   h(P(m,n))\leq 2n\sqrt{2\gamma(P(m,n))}\leq\frac{2\pi\sqrt{2n(n+1)}}{m}.
\]
Thus, for every fixed $n\geq2$, $h(P(m,n))\to 0$ as $m\to\infty$.

Furthermore, recall that a family $\{\mathcal{G}_j\}_{j\geq1}$ of finite $d$-regular graphs $\mathcal{G}_j=(V_j,E_j)$ with $|V_j|\to\infty$ is an \textit{expander family} if there exists $\epsilon>0$
such that the expansion ratio of every graph in the family is at least $\epsilon$. As a consequence, for every fixed $n>2$, the $2n$-regular family $\{P(m,n)\}_{m>2}$ is not an expander family. In the notation of Blanco and Buehrle~\cite{BB25}, the graph $P(m,n)$ is their undirected prefix-reversal graph $\mathbb P_m(n)$; therefore this disproves their Conjecture 6.3 from~\cite{BB25}.
\end{remark}

We finish this section with an example that illustrates the proof steps in the case $m=3$ and $n=2$. The Schreier graph is constructed, and then the operators $T_{n,r}$, from which the spectral gap is computed. 

\subsubsection{An example}

\begin{example} Let $G=S(3,2)$ and $H=S(3,1)$. We distinguish the symbol $2$.  The coset space $G/H$ is
naturally identified with $X_{3,2}=\mathbb Z_3\times[2]$, where a state $(a,p)$ records the color $a\in\mathbb Z_3$ and the position $p\in\{1,2\}$ of the distinguished symbol. Notice that in this case, the full Cayley graph has $3^2\times 2!=18$ vertices while the coset Schreier graph has only $3\times 2=6$ vertices.  

The set of generalized prefix reversals is $S=\{r_1^+,r_1^-,r_2^+,r_2^-\}$. Its action on $X_{3,2}$ is
\[
\begin{aligned}
   r_1^\pm(a,1)&=(a\pm1,1),
   &
   r_1^\pm(a,2)&=(a,2),\\
   r_2^\pm(a,1)&=(a\pm1,2),
   &
   r_2^\pm(a,2)&=(a\pm1,1),
\end{aligned}
\]
where the color additions and subtractions are performed modulo $3$.

Thus the three vertices $(0,1),(1,1),(2,1)$ form a triangle under the length-one reversals.  Each vertex
$(a,2)$ has two loops, one from $r_1^+$ and one
from $r_1^-$, and the length-two reversals join $(a,1)$ to $(a+1,2)$ and $(a-1,2)$. The corresponding Schreier graph is depicted in Figure~\ref{fig:schreier-m3-n2}.

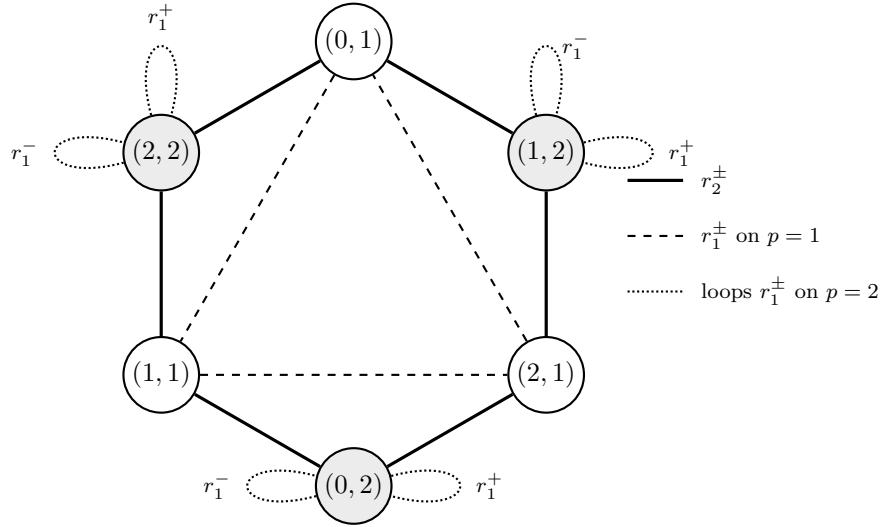
\begin{figure}[t]
\centering
\begin{tikzpicture}[
    scale=1.05,
    vertex/.style={
        circle,
        draw,
        thick,
        minimum size=10mm,
        inner sep=1pt,
        font=\small
    },
    positionone/.style={
        vertex,
        fill=white
    },
    positiontwo/.style={
        vertex,
        fill=gray!15
    },
    rtwo/.style={
        very thick
    },
    rone/.style={
        thick,
        dashed
    },
    holding/.style={
        thick,
        densely dotted
    }
]

\node[positionone] (a0) at ( 90:2.8) {$(0,1)$};
\node[positiontwo] (b1) at ( 30:2.8) {$(1,2)$};
\node[positionone] (a2) at (-30:2.8) {$(2,1)$};
\node[positiontwo] (b0) at (-90:2.8) {$(0,2)$};
\node[positionone] (a1) at (-150:2.8) {$(1,1)$};
\node[positiontwo] (b2) at (150:2.8) {$(2,2)$};

\draw[rtwo]
    (a0) -- (b1) -- (a2) -- (b0) -- (a1) -- (b2) -- (a0);

\draw[rone]
    (a0) -- (a1) -- (a2) -- (a0);

\path[holding]
    (b1) edge[
        loop right,
        min distance=12mm,
        looseness=6
     ] node[font=\scriptsize, right=2pt] {$r_1^+$} (b1);

\path[holding]
    (b1) edge[
        loop above,
        min distance=12mm,
        looseness=6
     ] node[font=\scriptsize, right=2pt] {$r_1^-$} (b1);

\path[holding]
    (b0) edge[
        loop right,
        min distance=12mm,
        looseness=6
     ] node[font=\scriptsize, right=2pt] {$r_1^+$} (b0);

    \path[holding]
    (b0) edge[
        loop left,
        min distance=12mm,
        looseness=6
     ] node[font=\scriptsize, left=2pt] {$r_1^-$} (b0);

\path[holding]
    (b2) edge[
        loop left,
        min distance=12mm,
        looseness=6
     ] node[font=\scriptsize, left=2pt] {$r_1^-$} (b2);

    \path[holding]
    (b2) edge[
        loop above,
        min distance=12mm,
        looseness=6
     ] node[font=\scriptsize, above=2pt] {$r_1^+$} (b2);

\draw[rtwo] (3.45,1.05) -- (4.15,1.05);
\node[anchor=west, font=\scriptsize] at (4.25,1.05)
    {$r_2^\pm$};

\draw[rone] (3.45,0.35) -- (4.15,0.35);
\node[anchor=west, font=\scriptsize] at (4.25,0.35)
    {$r_1^\pm$ on $p=1$};

\draw[holding] (3.45,-0.35) -- (4.15,-0.35);
\node[anchor=west, font=\scriptsize, align=left] at (4.25,-0.35)
    {loops $r^\pm_1$ on $p=2$};

\end{tikzpicture}

\caption{
The coset Schreier graph $\Sch(3,2)$ for $m=3$ and $n=2$. A vertex $(a,p)$ records the color $a\in\mathbb Z_3$ and position $p\in[2]$ of the distinguished symbol.
}
\label{fig:schreier-m3-n2}
\end{figure}

Let $\mathcal NA^{\mathrm{Sch}}_{3,2}$ denote the normalized adjacency operator of this Schreier graph.  Since there are four generators,
\[
   \mathcal{NA}^{\mathrm{Sch}}_{3,2}
   =
   \frac14
   \left(
      \rho_{X_{3,2}}(r_1^+)+\rho_{X_{3,2}}(r_1^-)
      +\rho_{X_{3,2}}(r_2^+)+\rho_{X_{3,2}}(r_2^-)
   \right),
\]
where $\rho_{X_{3,2}}(\sigma)$ denotes the permutation matrix corresponding to $\sigma$. Because the generating set is symmetric, we may write its action as
\[
\begin{aligned}
   \bigl(\mathcal{NA}^{\mathrm{Sch}}_{3,2}F\bigr)(a,1)
   &=
   \frac14\Bigl(
      F(a+1,1)+F(a-1,1)
      +F(a+1,2)+F(a-1,2)
   \Bigr),\\
   \bigl(\mathcal{NA}^{\mathrm{Sch}}_{3,2}F\bigr)(a,2)
   &=
   \frac14\Bigl(
      2F(a,2)+F(a+1,1)+F(a-1,1)
   \Bigr).
\end{aligned}
\]

Order the vertices as follows: $(0,1),(1,1),(2,1),(0,2),(1,2),(2,2)$.

If $C$ denotes the following circulant matrix,
\[
   C=
   \begin{pmatrix}
      0&1&1\\
      1&0&1\\
      1&1&0
   \end{pmatrix},\quad \text{ then }\quad\mathcal{NA}^{\mathrm{Sch}}_{3,2}
   =
   \frac{1}{4}
   \begin{pmatrix}
      C&C\\
      C&2I_3
   \end{pmatrix}.
\]
Explicitly,
\[
   \mathcal{NA}^{\mathrm{Sch}}_{3,2}
   =
   \frac14
   \begin{pmatrix}
      0&1&1&0&1&1\\
      1&0&1&1&0&1\\
      1&1&0&1&1&0\\
      0&1&1&2&0&0\\
      1&0&1&0&2&0\\
      1&1&0&0&0&2
   \end{pmatrix}.
\]

Let $\omega=e^{2\pi i/3}$. For $r\in\{0,1,2\}$, let 
$\displaystyle{
(J_rf)(a,p)=\frac{1}{\sqrt{3}}\omega^{ra}f(p)\text{ and } V_r=\operatorname{im}(J_r).
}$
Fourier decomposition on $\mathbb Z_3$ gives $\ell^2(\mathbb Z_3\times[2];\mathbb C)=V_0\oplus V_1\oplus V_2$. Moreover, each $V_r$ is invariant under
$\mathcal{NA}^{\mathrm{Sch}}_{3,2}$. 

Additionally, let $c_r=\cos\left(\frac{2\pi r}{3}\right)$. If $F(a,p)=(1/\sqrt{3})\omega^{ra}f(p)=(J_rf)(a,p)$, then 

\[(\mathcal{NA}^{\mathrm{Sch}}_{3,2}F)(a,p) =(1/\sqrt{3})\omega^{ra}(T_{2,r}f)(p)\quad
\text{ where }\quad
   T_{2,r}
   =
   \frac12
   \begin{pmatrix}
      c_r&c_r\\
      c_r&1
   \end{pmatrix}.
\]
Therefore, $\mathcal{NA}^{\mathrm{Sch}}_{3,2} \cong T_{2,0}\oplus T_{2,1}\oplus T_{2,2}$. Notice that $c_0=1$, and $c_1=c_2=-\frac{1}{2}$, and so 
\[
   T_{2,0}
   =
   \frac12
   \begin{pmatrix}
      1&1\\
      1&1
   \end{pmatrix}
\quad\text{ and }\quad
   T_{2,1}=T_{2,2}
   =
   \begin{pmatrix}
      -\frac14&-\frac14\\[1mm]
      -\frac14&\frac12
   \end{pmatrix}.
\]
The eigenvalues of $T_{2,0}$ are $1$ and $0$, while the eigenvalues
of each of $T_{2,1}$ and $T_{2,2}$ are
\[
   \frac{1+\sqrt{13}}8
   \qquad\text{and}\qquad
   \frac{1-\sqrt{13}}8.
\]
It follows that
\[
   \operatorname{Spec}
   \bigl(\mathcal{NA}^{\mathrm{Sch}}_{3,2}\bigr)
   =
   \left\{
1,\,0,\,\left(\frac{1+\sqrt{13}}8\right)^{[2]},\, \left(\frac{1-\sqrt{13}}{8}\right)^{[2]}\right\},
\] where the superscript $[2]$ is used to denote eigenvalues with multiplicity 2. Therefore the normalized Schreier spectral gap is
\[
   1-\lambda_2
   \bigl(\mathcal{NA}^{\mathrm{Sch}}_{3,2}\bigr)
   =
   \frac{7-\sqrt{13}}8.
\]
Since the Schreier graph has operator degree $4$, its unnormalized
gap is
\[
   \beta_{3,2}
   =
   4\left(\frac{7-\sqrt{13}}8\right)
   =
   \frac{7-\sqrt{13}}2.
\]
Finally, notice that the semi-recursive formula holds in this case. Indeed, since $P(3,1)$ is the 3-cycle, 
\[
   \psi_{3,1}
   =
   2-2\cos\left(\frac{2\pi}{3}\right)
   =
   3.
\]
Since $\displaystyle{\beta_{3,2}=\frac{7-\sqrt{13}}2<3}$,
Corollary~\ref{cor:cesi} gives
\[
   \frac{7-\sqrt{13}}2
   =
   \min\{\psi_{3,1},\beta_{3,2}\}
   \le
   \psi_{3,2}
   \le
   \beta_{3,2}
   =
   \frac{7-\sqrt{13}}2.
\]
Therefore,
$\displaystyle{
   \psi_{3,2}
   =
   \frac{7-\sqrt{13}}2.
}$ 

Since $P(3,2)$ is $4$-regular, its normalized spectral gap is
$\displaystyle{
   \gamma(P(3,2))
   =
   \frac{\psi_{3,2}}4
   =
   \frac{7-\sqrt{13}}8.
}$
\end{example}

\section{Concluding remarks}

\begin{itemize}
    \item We proved that the normalized Laplacian spectral gap of the generalized pancake graph $P(m,n)$ is $\Theta_m(1/n)$, for fixed $m\geq 2$, as $n\to\infty$. This extends what was already known in the literature for the pancake graph $P(1,n)$. The constant in the lower bound depends on $m$, and it approaches 0 as $m$ gets larger. We also proved that $\gamma(P(m,n))$ is $\Theta_n(m^{-2})$ for fixed $n$ as $m\to\infty$. As a consequence, the family $\{P(m,n)\}_{m>2}$, with fixed $n>2$, cannot be an expander family, which disproves a conjecture of Blanco and Buehrle~\cite[Conjecture 6.3]{BB25}.
    \item Regarding the \textit{adjacency} spectral gap of $P(m,n)$, an exact formula is known for $m=1$ and all $n\geq2$ (see Cesi~\cite{Cesi2009} and Chung and Tobin~\cite{ChungTobin2017}). Moreover, the case $n=1$ and $m\geq2$ is elementary. For general $m,n\geq2$, no closed formula is currently known for the adjacency spectral gap. It seems difficult to even conjecture such a formula given that numerical experiments do not seem to suggest a discernible pattern. 
    \item We showed that $\psi_{m,n}\leq \beta_{m,n}$. In other words, the spectral gap of the coset Schreier graph is at least that of the full Cayley graph. Following the release of an earlier draft of this paper, Qiyuan (Alex) Gu observed that strict inequality is possible; for example, direct computation yields $\psi_{6,2}=1<4-2\sqrt{2}=\beta_{6,2}$. As a consequence, a conjecture of Graves and Zhu asserting that equality holds for $m\geq2, n\geq1$ ~\cite[Conjecture 2]{Greaves2026} is false. Further numerical experiments with code provided by Gu allows us to make the following conjecture.
    \begin{conjecture}\label{con:ineq}
        $\psi_{m,n} = \beta_{m,n}$ if and only if (i) $m = 6$ and $n \geq 7$, or (ii)  $2 \leq m \leq 5$ and $n \geq 2$.
    \end{conjecture}
    \item Our spectral-gap estimate has a direct interpretation for the associated random walk.  Consider the \textit{lazy simple random walk} on $P(m,n)$. At each discrete time, it remains at its current state with probability $1/2$; otherwise, it applies a uniformly chosen element of $R(m,n)$ to the current state of the walk. As a consequence of the bound for the spectral gap, the \textit{relaxation time} of the lazy simple random walk is $\Theta_m(n)$ as $n\to\infty$ for fixed $m\geq2$. The same order holds for the coset Schreier graph that tracks the color-position pairs. In particular, the relaxation time grows linearly with the number of positions, which is perhaps unexpected for the full Cayley graph. 
    
    Moreover, since the stationary distribution is uniform, the standard spectral estimate yields the general mixing-time upper bound (see Levin and Peres~\cite[Theorem 12.4]{Levin2017})
\[
   t^{\mathrm{lazy}}_{\mathrm{mix}}(\epsilon)
   =O_m\!\left(n\log\frac{|V|}{\epsilon}\right),
\] where $V$ is the vertex set of the relevant Cayley graph or the Schreier graph. The spectral-gap estimate also yields a lower bound of order $n$ for fixed $\epsilon$; that is, $t_{\mathrm{mix}}^{\mathrm{lazy}}(\epsilon)=\Omega_m(n)$ as $n\to\infty$ for fixed $m\geq2$ and $\epsilon\in(0,1/2)$. We do not, however, have a matching asymptotic lower bound of order $n\log(|V|/\epsilon)$.

\end{itemize}

The present argument does not determine the exact spectral gap or the precise mixing-time scale, nor does it establish the exact asymptotic order of the Cheeger constant.  These questions provide natural directions for further study, aloing with Conjecture~\ref{con:ineq}.

\section*{Acknowledgment} The author thanks Qiyuan (Alex) Gu for pointing out that strict inequality can happen in~(\ref{eq:mn}). In particular, $\psi_{6,2}<\beta_{6,2}$. Gu also provided code that was helpful in formulating Conjecture~\ref{con:ineq}

\bibliographystyle{plain}

\end{document}